\documentclass[12pt, reqno]{amsart}
\usepackage{amssymb, amsthm, amsmath, amsfonts}
\usepackage{array, epsfig}
\usepackage[unicode]{hyperref}
\usepackage[numbers,square]{natbib}
\usepackage{color}

  \evensidemargin
\oddsidemargin
\newcommand{\dom}{\mathop{\mathrm{dom}}\nolimits}

\renewcommand{\phi}{\varphi}

\newcommand{\bF}{\mathbb{F}}

\newcommand{\bL}{\mathbb{L}}

\newcommand{\bQ}{\mathbb{Q}}

\newcommand{\bS}{\mathbb{S}}

\newcommand{\bH}{\mathbb{H}}

\newcommand{\cC}{\mathcal{C}}

\newcommand{\cG}{\mathcal{G}}
\newcommand{\cH}{\mathcal{H}}

\newcommand{\cO}{\mathcal{O}}
\newcommand{\cU}{\mathcal{U}}

\newcommand{\cT}{\mathcal{T}}
\newcommand{\cX}{\mathcal{X}}

\newcommand{\N}{\mathbb{N}}

\newcommand{\Z}{\mathbb{Z}}

\theoremstyle{plain}
\newtheorem{theorem}{Theorem}[section]
\newtheorem{lemma}[theorem]{Lemma}
\newtheorem{corollary}[theorem]{Corollary}
\newtheorem{proposition}[theorem]{Proposition}

\theoremstyle{definition}
\newtheorem{definition}[theorem]{Definition}

\theoremstyle{remark}
\newtheorem{remark}[theorem]{Remark}

\begin{document}

\author{Zakhar Kabluchko}
\address{Zakhar Kabluchko: Institut f\"ur Mathematische Stochastik,
Westf\"alische Wilhelms-Universit\"at M\"unster,
Orl\'eans-Ring 10,
48149 M\"unster, Germany}
\email{zakhar.kabluchko@uni-muenster.de}

\author{Katrin Tent}
\address{Katrin Tent:
Mathematisches Institut und
Institut f\"ur Mathematische Logik und Grundlagenforschung,
Westf\"alische Wilhelms-Universit\"at M\"unster,
Einsteinstr.\ 62,
48149 M\"unster,
Germany}
\email{tent@wwu.de}

\title[]{On weak Fra\"iss\'e limits}




\begin{abstract}
Using the natural action of $S_\infty$ we show that a countable hereditary  class $\cC$ of finitely generated structures has the joint embedding property (JEP) and the weak amalgamation property (WAP)
if and only if there is a structure $M$ whose isomorphism type is comeager in the space of all countable, infinitely generated structures with age in $\cC$. In this case, $M$ is the weak Fra\"iss\'e limit of $\cC$.

This applies in particular to countable structures with generic automorphisms and recovers a result by Kechris and Rosendal [\textit{Proc.\ Lond.\ Math.\ Soc.,\ 2007}].
\end{abstract}

\maketitle

\section{Introduction}

After writing the note~\cite{kabluchko_tent_comeager1} showing that the class of universal-homogenous structures is comeager in some appropriate space of structures, we noticed that the methods from the paper of Kechris and Rosendal~\cite{KechrisRosendal} easily yield a stronger and more general result applying to weak Fra\"iss\'e classes of finitely generated structures.

Thus, the aim of the present note is to prove that for an arbitrary hereditary class $\cC$ of finitely generated (not necessarily finite) structures, the space of all countable, infinitely generated structures with age in $\cC$ contains a comeager subset of structures isomorphic to some countable structure $M$ if and only if  $\cC$ has the weak Fra\"iss\'e property and $M$ is the weak Fra\"iss\'e limit of $\cC$.   As a corollary, this proves the existence and uniqueness of the weak Fra\"iss\'e limit of a weak Fra\"iss\'e class. We also explain how this in turn reproves the result by Kechris and Rosendal~\cite{KechrisRosendal} on comeager conjugacy classes of automorphisms of such structures in the slightly more general setting of finitely generated structures.

Similar results for classes of finite structures were obtained by Kruckman in his PhD thesis \cite{kruckman} using Banach--Mazur games on posets.


\section{Background and the main result}\label{sec:main}
\subsection{Weak Fra\"ss\'e limits}

Let $L$ be a countable language. For background on model theory we refer to \cite{TZ}.
Here and in what follows, the word `structure' always means `$L$-structure'.

\begin{definition}\label{def:Fraisse} Let $\cC$ be a  countable class of (isomorphism types of) finitely generated structures and assume that $\cC$ is hereditary, that is closed under taking finitely generated substructures.
We say that  $\cC$ is  a  \emph{weak Fra\"iss\'e class} if $\cC$ satisfies
\begin{itemize}
\item[(WAP)] (Weak Amalgamation Property) For any $A\in\cC$ there is some $B\in \cC$ and embedding $e:A\longrightarrow B$ such that for all $B_1, B_2\in\cC$ and embeddings $h_i$ of $B$ into $B_i, i=1,2$, there is some $C\in\cC$ and embeddings $g_i$ of $B_i$ into $C, i=1,2$, such that $g_1\circ h_1\circ e=g_2\circ h_2\circ e$.

\item[(JEP)] (Joint Embedding Property) For any $A, B\in\cC$ there is some $D\in\cC$ containing $A$ and $B$ as a substructure.
\end{itemize}

The class $\cC$ is called a  \emph{Fra\"iss\'e class}  if in condition (WAP) we can choose $B=A$ and $e=id$. The notion of Weak Amalgamation was first studied by Ivanov in \cite{ivanov}.

We call $\cC$ \emph{unbounded} if  any $A\in\cC$ can be embedded as a proper substructure into some $B\in \cC$.
\end{definition}

For a structure $M$, the \emph{age} of $M$, $age(M)$, is the class of isomorphism types of finitely generated substructures of $M$.

\begin{definition}\label{def:weakFraisse}
Let $\cC$ be a countable hereditary class of finitely generated structures
and let $M$ be a countable structure.

\begin{enumerate}

\item ($\cC$-Universality) We say that $M$ is $\cC$-universal if $age(M)=\cC$.

\item (Weak $\cC$-Homogeneity) We say that $M$ is weakly $\cC$-homogeneous if for any structure $A\in\cC$ there is some $B\in\cC$ and embedding $e: A\longrightarrow B$ such that for any
substructures $B_1, B_2\subset M$ with $(A_1, B_1), (A_2,B_2)$ isomorphic to $(e(A),B)$ any isomorphism $\phi: (A_1,B_1)\to (A_1,B_2)$ extends to an automorphism $f$ of $M$ with $\phi|_{A_1}=f|_{A_1}$.


\item (Weak $\cC$-Saturation) For any structure $A\in\cC$ there is some $B\in\cC$ and embedding $e: A\longrightarrow B$ such that for any $D\in\cC$ containing $B$  and any embedding $h: B\longrightarrow M$ there is an embedding of $D$ into $M$ extending $h\circ e$ on $A$.


\end{enumerate}

The structure $M$ is called $\cC$-homogeneous  ($\cC$-saturated, respectively) if  we can choose $A=B$ and $e=id$.

A countable structure $M$ is called a (weak) Fra\"iss\'e limit for $\cC$ if it is $\cC$-universal and (weakly) $\cC$-homogeneous.

\end{definition}

We will show that a countable hereditary unbounded class $\cC$ of finitely generated structures has a weak Fra\"iss\'e limit if and only if $\cC$  satisfies (JEP) and (WAP).

\subsection{The space of structures and the action of \texorpdfstring{$S_\infty$}{S infty}}
Let $\cC$ be a countable hereditary class of finitely generated structures
and let $\bS$ be the set of all structures $M$ defined on the universe $\omega$ (or any other fixed countable universe) with the following two properties:
\begin{itemize}
\item[(i)] the age of $M$ is contained in $\cC$ and
\item[(ii)] $M$ is not finitely generated.
\end{itemize}

\vspace*{2mm}
\noindent
{\bf Metric on $\bS$:} We endow $\bS$ with the following metric:
For two different structures $M, N\in\bS$ we put
$d(M,N) = 2^{-k}$
where $k\in\{0,1,\ldots\}$ is minimal such that the substructures of $M$ and $N$
generated by the elements $0, 1,\ldots, k$ are different, i.e.\ either these
substructures have different universes as subsets of $\omega$ or their universes agree, but some constant, relation or function differs on the universe.
Note that with this metric, $\bS$ becomes a complete (ultra)metric space. Consequently, the Baire category theorem holds in $\bS$.

\vspace*{2mm}
\noindent
{\bf Topology on $\bS$:} For a finitely generated structure $B$ whose isomorphism type is contained in $\cC$ and such that $\dom B\subset \omega$,
let $\cO_B$ be the set of all structures $A\in \bS$ whose restriction to $\dom B$ coincides with $B$. It is easy to see that these sets (after discarding the empty ones) form a basis of some topology $\cT$ on $\bS$, see~\cite{kabluchko_tent_comeager1}, and that $\cT$ is the topology induced by the metric $d$ on $\bS$.

\begin{proposition}
The space $(\bS,\cT)$ is separable if and only if every structure from $\cC$ that can be embedded into a countable, infinitely generated structure,  is finite.
\end{proposition}
\begin{proof}
Let  $(\bS,\cT)$ be separable with a countable dense set $\{X_1,X_2,\ldots\}$. Assume, by contraposition, that there is an infinite structure $A_0\in \cC$ that can be embedded into some countable infinitely generated structure.   Taking some infinite set $U\subset \omega$ as a universe, we can construct a structure $A(U)$ isomorphic to $A_0$. The open sets of the form $\cO_{A(U)}$ are non-empty, hence any such set contains some point of the form $X_i$. Since there are uncountably many possible $U$'s, there is at least one $X_i$ contained in uncountably many sets of the form $\cO_{A(U)}$. Hence, $X_i$ has uncountably many finitely generated substructures with pairwise different universes. But this is a contradiction. Conversely, if all finitely generated substructures of any $M\in\bS$ are finite, then there are only countably many non-empty basic open sets of the form $\cO_B$. By choosing a point in any such set, we obtain a countable dense subset of $\bS$, thus proving separability.
\end{proof}

\vspace*{2mm}
\noindent
{\bf Action of $S_\infty$ on  $\bS$:}  Consider $S_\infty$, the permutation group on $\omega$, as a Polish group with  a neighbourhood basis of the identity consisting of pointwise stabilizers of finite sets.
There is a natural continuous action of   $G=S_\infty$ on  $\bS$ just given by the action $g\in G$ on the domain $\omega$ of a structure and transferring the structure to the image. Under this action, the orbit of a structure $M\in\bS$ is the class of structures isomorphic to $M$ and the stabilizer $G_M$ consists of all automorphisms of $M$.

\begin{proposition}\label{prop:turbulent}(\cite{KechrisRosendal}, Prop. 3.2)
Let $G$ be a closed subgroup of $S_\infty$ acting continuously on a topological space $X$ and let $x\in X$. Then the following are equivalent:
\begin{enumerate}
\item The orbit of $x$ under $G$ is non-meager.
\item For any open subgroup $V$ of $G$ the orbit of $x$ under $V$ is
somewhere dense.
\end{enumerate}
\end{proposition}
\begin{proof}
$\neg(2)\Rightarrow \neg(1)$:
Suppose that the orbit  of $x$ under some open subgroup $V\leq G$ is nowhere dense. Since $V$ has countable index in $G$ and the orbit of $x$ under each coset $gV$ is also nowhere dense, the orbit of $x$ under $G$ is meager.

\vspace*{2mm}
\noindent
$\neg (1)\Rightarrow \neg(2)$:
Conversely, assume that the orbit of $x$ under $G$ is meager which means that it is contained in the countable union of closed nowhere dense sets $F_n, n\in\omega$. Then the sets $K_n=\{g\in G: g(x)\in F_n\}$ are closed in $G$ and their union is all of $G$. At least one of the $K_n$ has non-empty interior, so contains a translate of some open subgroup $V$. Thus the orbit of $x$ under (all translates of) $V$ is nowhere dense.
\end{proof}

\subsection{Main result}
Recall that a subset of a topological space is called comeager if its complement is meager or, equivalently, if the set can be represented as a countable intersection of subsets with dense interior.
\begin{theorem}\label{thm:main}
Let $\cC$ be a countable hereditary unbounded\footnote{An example due to M\'ark Po\'or shows
that the assumption that $\cC$ be unbounded is necessary.
} class of finitely generated structures and let  $(\bS, \cT)$ be the topological space of all countable structures on $\omega$ with age contained in $\cC$  and which are not themselves finitely generated. The topology $\cT$ is given by the metric $d$ as above. Then the following
are equivalent:
\begin{enumerate}
\item The class $\cC$ has a weak Fra\"iss\'e limit $M\in\bS$.
\item There is a $\cC$-universal and weakly $\cC$-saturated structure $M\in\bS$.
\item The $S_\infty$-orbit of some structure $M\in\bS$ is comeager in $\bS$.
\item For some structure $M\in\bS$  the set of structures isomorphic to $M$ is comeager in $\bS$.
\item The class $\cC$ satisfies (JEP) and (WAP).
\item There is a $\cC$-universal weakly $\cC$-saturated structure $M\in\bS$ such that (*) holds:
\begin{itemize}

 \item[(*)]  any finitely generated substructure $A$ of $M$ is contained in a finitely generated substructure $B\subset M$ (not necessarily the one given by  weak $\cC$-saturation) such that any $D\in\cC$ containing a copy of $B$ embeds over $A$ into $M$.

\end{itemize}
\end{enumerate}
It follows in particular, that a weak Fra\"iss\'e limit, if it exists, is unique up to isomorphism.

\end{theorem}

\begin{proof}
(1) $\Rightarrow$ (2):
 Let $M\in\bS$ be a weak Fra\"iss\'e limit of $\cC$. We have to show that $M$ is weakly $\cC$-saturated, Let $A\subseteq B\in\cC$  be as given by weak $\cC$-homogeneity. Assume that $(A_0,B_0)$ is an isomorphic copy of $(A,B)$ in $M$.
Let $D\in\cC$ contain an isomorphic copy of $B$. By $\cC$-universality,  $M$ contains a substructure $D_0$ isomorphic to $D$. By $\cC$-homogeneity, any isomorphism between the copy of $A$ inside $D_0$ and $A_0$ can be extended to an automorphism of  $M$, yielding an embedding of $D$ over $A_0$ into $M$ as required.

(2) $\Rightarrow$ (3):
Let $M\in\bS$ be $\cC$-universal and weakly $\cC$-saturated. Then $age(M)=\cC$ and thus the $S_\infty$-orbit of $M$ is dense. It now suffices to verify that the $S_\infty$-orbit of $M$ is non-meager: since the action of $S_\infty$ on the metric space $\bS$ is continuous, it then follows from \cite{Sami}, Thm. 4.4 that the orbit is $G_\delta$ and hence comeager (given that it is dense). To see that the $S_\infty$-orbit of $M$ is non-meager, it suffices by Proposition~\ref{prop:turbulent} to see that for any open subgroup $V\leq S_\infty$ the $V$-orbit of $M$ is somewhere dense. We may assume that there is a finite set $A_0$ such that $V=G_{A_0}$ is the pointwise stabilizer of $A_0$. Let $A\subset M$ be the substructure generated by $A_0$ and let $B\in\cC$ be such that $B$ contains an isomorphic copy of $A$ and such that any extension of $B$ can be embedded into $M$ by
weak $\cC$-saturation. Now since $B\in age(M)$ we may assume that $B\subset M$. Let $A'$ denote the isomorphic copy of $A$ inside $B$ with $A_0'$ denoting the finite set corresponding to $A_0$. Then $V'=G_{A_0'}$ is conjugate to $G_{A_0}$ in $S_\infty$ and it clearly suffices to prove the claim for $V'$.
We claim that the $V'$-orbit of $M$ is dense in $O_B$:
By $\cC$-saturation of $M$, any extension $D$ of $B$ can be embedded over $A'$ into $M$. This says that for some $g\in V'$ we have $g(M)\in O_D$, which is enough.

(3) $\Leftrightarrow$ (4): clear.

(3) $\Rightarrow$ (5):
Let $M\in\bS$ be an element whose orbit under $G=S_\infty$ is comeager.  Since
$\bS$ is a Baire space, a comeager set is dense and thus we see that $age(M)=\cC$ and hence (JEP) holds. It remains to prove that $\cC$ satisfies (WAP). Let $A\in\cC$. Then $M$ contains a substructure isomorphic to $A$ and we may  assume that $A\subset M$. Let $A_0\subset A$ be a finite set generating $A$. Then the pointwise stabilizer of $A_0$ in $G$ is an open subgroup $G_{A_0}$ and by Proposition~\ref{prop:turbulent} the orbit of $M$ under $G_{A_0}$ is dense in some basic open set $O_B$ for some finitely generated structure $B$. Thus, $M':= f(M)\in \cO_B$ for some $f\in G_{A_0}$.  Let $A'=f(A)$ be the structure (with domain contained in $\omega$) obtained by transferring $A$ by $f$. Then, $M'\in \cO_B \cap \cO_{A'}$ meaning that $B$ and $A'$ are substructures of $M'$.
Let $D$ be the substructure of $M'$ generated by $A'$ and $B$.  Then, $O_D\subset O_{A'}\cap O_B$.
To prove (WAP), let $E, F\in\cC$ be structures containing $D$ (and assume that the domains of $E$ and $F$ are embedded into $\omega$). Since the orbit of $M$ (and hence $M'$) under $G_{A_0}$ is dense in $O_D$, there exist $g, h\in G_{A_0}$ such that $g(M')\in O_E$
and  $h(M')\in O_F$. Observe that both $g$ and $h$ equal to the identity map on $A_0$, hence $g|_{A'}=h|_{A'} = id$.  Hence the substructure of $M'$ generated by $g^{-1}(E)\cup h^{-1}(F)\subset M'$ is the required weak amalgam over $A'$.

(5) $\Rightarrow$ (6):
Assume that  $\cC$ has (JEP) and (WAP). We will construct  a
$\cC$-universal and weakly $\cC$-saturated structure satisfying (*). 
  Choose an enumeration $(C_i)_{i\in\omega}$ of all isomorphism
  types in $\cC$. We construct $M$ as the union of an ascending
  chain
  \[M_0\subset M_1\subset\cdots\subset M\]
  of elements of $\cC$. Suppose that $M_{i}$ is already constructed.
  If $i=2n$ is even, let $A_{i+1}$ be the structure  embedding both $M_i$ and $C_n$
  given by (JEP) and let $M_{i+1}\in\cC$ be the structure given by (WAP) for $A_{i+1}$.

  We use the odd steps to ensure that the limit $M$ is weakly $\cC$-saturated.
  For $i=2n+1$, let $A\in\cC$ and let $B\in\cC$  be
  as guaranteed by (WAP) where we assume that $A\subset B$.
  Let $D\in\cC$ and assume we are
  given two embeddings
   $f_0\colon B\to M_i$ and $f_1\colon B\to D$. Then we let $M_{i+1}\in\cC$ be
   the structure guaranteed to exist by (WAP).

  Assume now
  that we have $A\subset B,D \in\cC$ and embeddings $f_0\colon B\to D$
  and $f_1\colon B\to M$.  Since $B$ is finitely generated, the image
  of $f_0$ will be
  contained in some $M_j$. Thus, in order to guarantee the weak $\cC$-saturation
  of $M$, we have to
  ensure during the construction of the the $M_i$,
  that eventually, for some odd $i\geq j$, the embeddings
  $f_0\colon B\to M_i$ and $f_1\colon B\to D$ were used in the
  construction of $M_{i+1}$. This can be done since for each $j$
  there are -- up to isomorphism -- at most countably many
  possibilities.  Thus there exists an embedding $g_1\colon
  D\to M_{i+1}$ with $f_0|_A=g_1\circ f_1|_A$.

By the even stages, we have $\cC\subset  age(M)$. Conversely, the finitely generated
substructures of $M$ are the substructures of the
  $M_i$. Since the $M_i$ belong to $\cC$, their finitely-generated
  substructures also belong to $\cC$. Hence $age(M)=\cC$ as required.
  Furthermore, since any finitely generated substructure $A$ is contained in some $M_i$,
  the structure $M_{2i}$ satisfies (*) for~$A$. Note that $M$ is not finitely generated as otherwise we would have $M\in\cC$ and hence $\cC$ would be bounded.

(6) $\Rightarrow$ (1):
Let $M\in\bS$ be $\cC$-universal and
 weakly $\cC$-saturated  such that (*) holds.
 We have to show that $M$ is weakly $\cC$-homogeneous. Let $A_0\subseteq A_1\in\cC$  be as given by weak $\cC$-saturation. Assume that $(B_0,B_1)$ and $(C_0,C_1)$ are isomorphic copies of $(A_0,A_1)$ in $M$ and that $\phi_0: (B_0,B_1)\to (C_0,C_1)$ is an isomorphism. We want to extend $\phi_0|_{B_0}$ to an automorphism of $M$ by a  back-and-forth construction.
Suppose we have already constructed $\phi_{2i-1}: B_{2i-1}\to C_{2i-1}$ with $\phi_{2i-1}|_{B_0}=\phi_0$
such that  $C_{2i-1}$ is an extension of $C_{2i-2}$ satisfying (*).
 Let $B_{2i}$ be the substructure given by (*) containing the substructure of $M$ generated by $B_{2i-1}\cup\{i\}$. Then
 there is a copy $C_{2i}$ such that there is an isomorphism $\phi_{2i}: (B_{2i-1}, B_{2i})\to (C_{2i-1},C_{2i})$ with $\phi_{2i}|_{B_0}=\phi_0$. Now let  $C_{2i+1}$ be the structure given by (*) containing  the substructure of $M$ generated by $C_{2i}\cup\{i\}$ and choose an isomorphic copy  $B_{2i+1}$  of $C_{2_i+1}$ and an
 isomorphism $\phi_{2i+1}: (B_{2i}, B_{2i+1}) \to (C_{2i},C_{2_i+1})$ extending $\phi_{2i}|_{B_{2i-1}}$.
 Thus again $\phi_{2i+1}|_{B_0}=\phi_0$ and we continue.
\end{proof}

\subsection{Finding automorphisms with comeager conjugacy class}

The proof of Theorem~\ref{thm:main} was motivated by the methods from \cite{KechrisRosendal}. We now show how this in turn implies their result on comeager conjugacy classes.

For a Fra\"iss\'e class $\cC$ of finitely generated $L$-structures, let $\cC_p$ denote the class of all  systems of the form $S=\langle A, \psi:B\to C\rangle$ where $A, B, C\in\cC, B, C\subset A$ and $\psi$ is an isomorphism of $B$ and $C$. An embedding of one system  $S=\langle A, \psi:B\to C\rangle$  into
 another system  $T=\langle D, \phi:E\to F\rangle$
is an embedding $f \colon A\to D$ such that $f$ embeds $B$ into $E$, $C$ into $F$  and such that  $f\circ \psi\subset  \phi\circ f$.

We now consider comeagerness in the space $\bS_p$ of structures $(M,h)$ where $M\in \bS$ and
$h\in Aut(M)$. Then,  for any finitely generated substructure $A$ of $M$
and isomorphic finitely generated substructures $B\subset A$, $C=h(B)$, the system  $S=\langle A, h|_B : B\to C\rangle$ is in $\cC_p$. As before, the space $\bS_p$ is a complete metric space. The group $Aut(M)$ is endowed with the usual topology of pointwise convergence.

\begin{remark}
We could instead consider an expanded language $L_p=L\cup\{f\}$ containing a new symbol $f$ for a partial function. Then $\cC_p$ denotes the class of all structures $A\in\cC$ expanded by an interpretation of the symbol $f$
as an isomorphism $f: B\to C$  where $A, B, C\in\cC $ and $B, C\subset A$.  By embedding an $L_p$-structure $A$ into an $L_p$-structure $D$, we may enlarge the domain of the function $f$ inside $A$ (contrary to the usual conventions).

\end{remark}

Note that we do not assume the Fra\"iss\'e class $\cC$ to consist of finite structures.

\begin{theorem}\label{thm:conjugacyclass} Let $\cC$ be a Fra\"iss\'e class of finitely generated structures with Fra\"iss\'e limit $M$. Then the following are equivalent:
\begin{enumerate}
\item $\cC_p$ has (JEP) and (WAP).
\item The weak Fra\"iss\'e limit of $\cC_p$  exists in $\bS_p$ and has comeager isomorphism type in $\bS_p$.
\item There is $f\in Aut(M)$ with comeager conjugacy class in $Aut(M)$.
\end{enumerate}
If any of the above conditions holds, the weak Fra\"iss\'e limit of $\cC_p$ has the form $(M, f)$, where $f\in Aut(M)$ has a comeager conjugacy class.
\end{theorem}
\begin{proof}
(1) $\Rightarrow$ (2):
The proof of  Theorem~\ref{thm:main} (5) $\Rightarrow$ (6) shows how to construct the weak Fra\"iss\'e limit $(N, g)$.
 It is left to show that $g$ is an automorphism of $N$. For this it suffices to verify that $g$ is defined on all of $N$ and is surjective. Suppose not, so there is some $a\in N$ where $g$ is not defined. Since $N$ is the weak Fra\"iss\'e limit of $\cC_p$ it is the countable union of systems $S_i=\langle N_i, \phi_i: B_i\to C_i\rangle$ for $i<\omega$. Let $i$ be minimal such that $g$ is not defined on all of $N_i$. By weak $\cC_p$-saturation, there is some system  $T=\langle D, \phi:E\to F\rangle$  extending $S_i$ such that any extension of $T$ can be embedded over $S_i$ into $N$. Since $\cC$ is a Fra\"iss\'e class with Fra\"iss\'e limit $M$, there is an embedding of $D$ into $M$ and by ultrahomogeneity of $M$, $\phi$ extends to an automorphism of $\hat{\phi}$ of $M$.
Then $\hat{\phi}$ is defined on all of $N_i\subset D$. Let $G$ be the substructure of $M$ generated by $D$ and $\hat{\phi}(D)$.  The system $\langle G, \hat{\phi}|_D: D\to \hat{\phi}(D)\rangle$ is in $\cC_p$ and extends $T$, thus embeds over $N_i$ into $N$, showing that $g$ is defined on all of $N$. A similar argument shows that $g$ is surjective.

The proof of Theorem~\ref{thm:main} (2) $\Rightarrow$ (3) shows
that the isomorphism type of $(N,g)$ in $\bS_p$ is comeager.

(2) $\Rightarrow$ (3):
Suppose the class $\cC_p$ has a weak Fra\"iss\'e limit $(N, g)$. 
 To see that $N$ is isomorphic to $M$ it suffices to check that $N$ is weakly $\cC$-saturated. Then both $M$ and $N$ have comeager isomorphism classes and thus must be isomorphic. Clearly we have $age(N)=\cC$. To see that $N$ is weakly $\cC$-saturated, let $A\subset N$ be a finitely generated substructure, so $A\in\cC$. Consider the system $\langle A, id: E\to E\rangle$ where $E$ is the substructure generated by the emptyset.  Note that any partial isomorphism extens $id: E\to E$. By (WAP) for $\cC_p$ there is some system $T=\langle B, \phi: C\to D\rangle$ such that any system containing $T$ embeds into $N$. Now let $F\in\cC$ be a finitely generated structure containing the structure $B$. Then
the system $\langle F, \phi: C\to D\rangle$ is in $\cC_p$ and hence embeds into $N$ over $A$. This shows that we can choose $B$ to witness the weak $\cC$-saturation of $N$.

It is left to prove that the conjugacy class of $g$ is comeager in $Aut(N)$. To see this note that the orbit of $(N,g)$ under $S_\infty$ is comeager in the class $\bS_p$. Since $g\in Aut(N)$, the
orbit of $S_\infty$ under the stabilizer of $N$ is exactly the orbit of $g$ under conjugation in $Aut(N)$. Since $Aut(N)$ is a Polish group, the claim follows from the Kuratowski-Ulam theorem.

(3) $\Rightarrow$ (1): The same proof as Theorem~\ref{thm:main} (5) $\Rightarrow$ (6) and (6) $\Rightarrow$ (1) shows the required.

\end{proof}

{\bf Acknowledgement:} We thank Tam\'as K\'atay, M\'ark Po\'or and Aristotelis Panagiotopoulos for pointing out the shortcomings of one of our previous definitions.





\bibliography{hall_group_bib}

\begin{thebibliography}{6}
\providecommand{\natexlab}[1]{#1}
\providecommand{\url}[1]{\texttt{#1}}
\expandafter\ifx\csname urlstyle\endcsname\relax
  \providecommand{\doi}[1]{doi: #1}\else
  \providecommand{\doi}{doi: \begingroup \urlstyle{rm}\Url}\fi

\bibitem[{Ivanov}(1999)]{ivanov}
A.A. {Ivanov}.
\newblock {Generic expansions of $\omega$-categorical structures and semantics
  of generalized quantifiers.}
\newblock \emph{{J. Symb. Log.}}, 64\penalty0 (2):\penalty0 775--789, 1999.
\newblock \doi{10.2307/2586500}.

\bibitem[Kabluchko and Tent(2017)]{kabluchko_tent_comeager1}
Z.~Kabluchko and K.~Tent.
\newblock Universal-homogeneous structures are generic.
\newblock \emph{Preprint}, 2017.
\newblock Available at \url{http://arxiv.org/abs/1710.06137}.

\bibitem[Kechris and Rosendal(2007)]{KechrisRosendal}
A.~S. Kechris and C.~Rosendal.
\newblock Turbulence, amalgamation, and generic automorphisms of homogeneous
  structures.
\newblock \emph{Proc. Lond. Math. Soc. (3)}, 94\penalty0 (2):\penalty0
  302--350, 2007.
\newblock ISSN 0024-6115.
\newblock URL \url{https://doi.org/10.1112/plms/pdl007}.

\bibitem[Kruckman(2016)]{kruckman}
A.~Kruckman.
\newblock Infinitary limits of finite structures.
\newblock \emph{PhD Thesis, University of California, Berkeley}, 2016.
\newblock {Available at
  \url{https://escholarship.org/content/qt2057n08h/qt2057n08h.pdf}}.

\bibitem[Sami(1994)]{Sami}
R.~L. Sami.
\newblock Polish group actions and the {V}aught conjecture.
\newblock \emph{Trans. Amer. Math. Soc.}, 341\penalty0 (1):\penalty0 335--353,
  1994.
\newblock ISSN 0002-9947.
\newblock URL \url{https://doi.org/10.2307/2154625}.

\bibitem[Tent and Ziegler(2012)]{TZ}
K.~Tent and M.~Ziegler.
\newblock \emph{A course in model theory}, volume~40 of \emph{Lecture Notes in
  Logic}.
\newblock Association for Symbolic Logic, La Jolla, CA; Cambridge University
  Press, Cambridge, 2012.
\newblock \doi{10.1017/CBO9781139015417}.
\newblock URL \url{http://dx.doi.org/10.1017/CBO9781139015417}.

\end{thebibliography}

\bibliographystyle{plainnat}

\end{document}